\documentclass[11pt,twoside,reqno]{amsart}

\usepackage{bbm}
\usepackage{verbatim}
\usepackage{subcaption}
\usepackage{graphicx,xcolor}
\usepackage{enumerate}
\usepackage[all]{xy}
\usepackage[hyperindex=true,plainpages=false,colorlinks=false,pdfpagelabels]{hyperref}

\newtheorem{theorem}{Theorem}[section]
\newtheorem{proposition}[theorem]{Proposition}
\newtheorem{lemma}[theorem]{Lemma}

\theoremstyle{remark}

\newtheorem{definition}[theorem]{Definition}

\numberwithin{equation}{section}

\DeclareMathOperator{\SL}{SL}

\DeclareMathOperator{\Tr}{tr}

\newcommand{\dvector}[1]{{\left(\begin{matrix}#1\end{matrix}\right)}}
\DeclareMathOperator{\dbar}{\overline\partial}

\newcommand{\R}{\mathbb{R}}
\newcommand{\C}{\mathbb{C}}
\newcommand{\N}{\mathbb{N}}
\newcommand{\Z}{\mathbb{Z}}
\begin{document}

\title[Irreducible flat ${\rm SL}(2,\R)$-connections on trivial bundle]{Irreducible flat ${\rm 
SL}(2,\R)$-connections on the trivial holomorphic bundle}

\author[I. Biswas]{Indranil Biswas}

\address{School of Mathematics, Tata Institute of Fundamental Research,
Homi Bhabha Road, Mumbai 400005, India}

\email{indranil@math.tifr.res.in}

\author[S. Dumitrescu]{Sorin Dumitrescu}

\address{Universit\'e C\^ote d'Azur, CNRS, LJAD, France}

\email{dumitres@unice.fr}

\author[S. Heller]{Sebastian Heller}

\address{Institute of Differential Geometry,
Leibniz Universit\"at Hannover,
Welfengarten 1, 30167 Hannover, Germany}

\email{seb.heller@gmail.com}



\date{\today}

\begin{abstract}
We construct an irreducible holomorphic connection with ${\rm SL}(2,\R)$--monodromy on the trivial
holomorphic vector bundle of rank two over a compact Riemann surface. This answers a question of 
Calsamiglia, Deroin, Heu and Loray in \cite{CDHL}.
\medskip

\noindent
\textsc{R\'esum\'e.}\, Dans cet article nous munissons le fibr\'e vectoriel holomorphe trivial de rang deux au-dessus 
d'une surface de Riemann compacte de genre $g \geq 2$, d'une connexion holomorphe irr\'eductible dont la monodromie 
est contenue dans ${\rm SL}(2,\R)$. Ceci r\'epond \`a une question pos\'ee par Calsamiglia, Deroin, Heu et Loray 
dans \cite{CDHL}.
\end{abstract}

\maketitle

\section{Introduction}\label{sec:intro}

Take a compact connected oriented topological surface $S$ of genus $g$, with $g \geq 2$. There is a natural
bijection between the isomorphism classes of
flat $\rm{SL}(2, \C)$--connections over $S$ and the conjugacy classes of group 
homomorphisms from the fundamental group of $S$ into $\rm{SL}(2, \C)$ (two such homomorphisms are called
conjugate if they differ by an inner automorphism of $\rm{SL}(2, \C)$). This bijection sends a flat connection to
its monodromy representation. When $S$ is equipped with a complex structure, a flat $\rm{SL}(2, \C)$--connection 
on $S$ produces a holomorphic vector bundle of rank two and trivial determinant on the Riemann surface 
defined by the complex structure on $S$; this is because locally constant transition functions producing
the vector bundle are holomorphic. In fact, since a holomorphic connection on a compact Riemann surface
$\Sigma$ is automatically flat, there is a natural bijection between the following two:
\begin{enumerate}
\item isomorphism classes of flat $\rm{SL}(2, \C)$--connections on a compact Riemann surface $\Sigma$;

\item isomorphism classes of pairs of the form $(E,\, D)$, where $E$ is a holomorphic vector bundle of rank two
on $\Sigma$ with $\bigwedge^2 E$ holomorphically trivial, and $D$ is a holomorphic connection on $E$ that induces
the trivial connection on $\bigwedge^2 E$.
\end{enumerate}
The above bijection is a special case of the Riemann--Hilbert correspondence \cite{De}.

Consider the flat $\rm{SL}(2, \C)$--connections on a compact Riemann surface $\Sigma$
satisfying the extra condition that the corresponding holomorphic vector bundle of rank two on $\Sigma$
is holomorphically trivial; they are known as differential ${\mathfrak 
s}{\mathfrak l}(2, \C)$--systems on $\Sigma$ (see \cite{CDHL}), where ${\mathfrak s}{\mathfrak l}(2, \C)$ is the Lie 
algebra of $\rm{SL}(2, \C)$. The differential ${\mathfrak s}{\mathfrak l}(2, \C)$--systems on $\Sigma$ are
parametrized by the complex vector space ${\mathfrak s}{\mathfrak l}(2, \C)\otimes_{\mathbb C}
H^0(\Sigma ,\, K_{\Sigma})$, where $K_{\Sigma}$ is 
the holomorphic cotangent bundle of $\Sigma$. A differential ${\mathfrak s}{\mathfrak l}(2, \C)$--system is called
irreducible if the monodromy representation of the corresponding flat connection is irreducible.
We shall now describe a context in which irreducible differential ${\mathfrak s}{\mathfrak 
l}(2, \C)$--systems appear.

For any cocompact lattice $\Gamma\, \subset\, {\rm SL}(2,{\mathbb C})$, the compact complex threefold ${\rm 
SL}(2,{\mathbb C}) / \Gamma$ does not admit any compact complex hypersurface \cite[p.~239, Theorem 2]{HM}.
While ${\rm SL}(2,{\mathbb C}) / \Gamma$ does not contain a ${\mathbb C}{\mathbb P}^1$, it may contain 
some elliptic curves. A question of Margulis asks whether ${\rm SL}(2,{\mathbb 
C})/\Gamma$ can contain a compact Riemann surface of genus bigger than one. Ghys has the following 
reformulation of Margulis' question: Is there a pair $(\Sigma,\, D)$, where $D$ is a differential ${\mathfrak 
s}{\mathfrak l}(2, \C)$--system on a compact Riemann surface $\Sigma$ of genus at least two, such that the 
image of the monodromy homomorphism $\pi_1(\Sigma)\, \longrightarrow\, \rm{SL}(2, \C)$ for $D$
is a conjugate of $\Gamma$? Existence of such a pair $(\Sigma,\, D)$ is equivalent to the existence of a
holomorphic map $\psi\, :\, \Sigma\, \longrightarrow\,{\rm SL}(2,{\mathbb C}) / \Gamma$ such that the
homomorphism $\psi_*\, :\, \pi_1(\Sigma)\, \longrightarrow\,\pi_1({\rm SL}(2,{\mathbb C}) / \Gamma)$ 
is surjective.

Being inspired by Ghys' strategy, the authors of \cite{CDHL} study the Riemann--Hilbert mapping for 
the irreducible differential ${\mathfrak s}{\mathfrak l}(2, \C)$--systems (see also 
\cite{BD}). Although some (local) results were obtained in \cite{CDHL} and \cite{BD}, the 
question of Ghys is still open. In this direction, it was asked in \cite{CDHL} (p. 161) 
whether discrete or real subgroups of $\rm{SL}(2, \C)$ can be realized as the monodromy of 
some irreducible differential ${\mathfrak s}{\mathfrak l}(2, \C)$--system on some compact 
Riemann surface. Note that if the flat connection on a compact Riemann surface $\Sigma$ corresponding
to a homomorphism $\pi_1(\Sigma)\, \longrightarrow\,
{\rm SL}(2,{\mathbb C})$ with finite image is irreducible, then the underlying holomorphic vector
bundle is stable \cite{NSe}, in particular, the underlying holomorphic vector bundle is not holomorphically trivial.

Our main result (Theorem \ref{Main}) is the construction of a pair $(\Sigma, \, D)$,
where $\Sigma$ is a compact Riemann surface of genus bigger than one and
$D$ is an irreducible differential ${\mathfrak s}{\mathfrak l}(2, \C)$--system on $\Sigma$, such that
the image of the monodromy representation for $D$ is contained in $\SL(2,\R)$.

In Section \ref{sec:prem} we collect preliminaries about moduli spaces and parabolic bundles. In Section 
\ref{sec:4-pun} we construct flat connections $\nabla$ with $\SL(2,\R)$-monodromy on prescribed parabolic bundles 
over a 4-punctured torus, and in Section \ref{sec:maiN} we show how $\nabla$ gives rise to an irreducible 
${\mathfrak s}{\mathfrak l}(2, \C)$--system with real monodromy on certain ramified coverings of the torus,
such that the underlying rank two holomorphic bundle is trivial.

\section{Preliminaries}\label{sec:prem}
\subsection{The Betti moduli space of a 1-punctured torus}\label{ADHS1t}

Let $\Gamma\,=\, {\mathbb Z}+\sqrt{-1}{\mathbb 
Z}\,\subset\, \mathbb C$ be the standard lattice. Define the elliptic curve $T^2\,:=\,\C/\Gamma$, and fix the point 
\begin{equation}\label{eo}
o\,=\,\left[\tfrac{1+\sqrt{-1}}{2}\right]\,\in\, T^2\, .
\end{equation}
For a fixed $\rho\,\in \, [0,\, \tfrac{1}{2}[$, we are interested in the Betti moduli space 
$\mathcal M^\rho_{1,1}$ parametrizing flat $\SL(2,\C)$--connections on the complement $T^2\setminus\{o\}$ whose 
local monodromy around $o$ lies in the conjugacy class of
\begin{equation}\label{locmon}\dvector{ \exp(2\pi \sqrt{-1} \rho) &0 \\ 0& \exp(-2\pi\sqrt{-1} \rho)}\,\in\,
{\rm SL}(2,{\mathbb C})\, .
\end{equation}
This Betti moduli space $\mathcal M^\rho_{1,1}$ does not depend on the
complex structure of $T^2$.
When $\rho\,=\,0$, it is the moduli space of flat $\SL(2,\C)$--connections on $T^2$; in that case
$\mathcal M^\rho_{1,1}$ is a
singular affine variety. However, for every $0\,<\,\rho\,<\,\tfrac{1}{2}$, the space
$\mathcal M^\rho_{1,1}$ is a nonsingular affine variety. We shall recall an explicit description of
this affine variety.

Let $x,\, y,\, z$ be the algebraic functions on $\mathcal M^\rho_{1,1}$ defined as follows: for any
homomorphism
\[h\,\colon\, \pi_1(T^2\setminus\{o\},\, q)\,\longrightarrow\, \SL(2,\C)\]
representing $[h]\,\in \, {\mathcal M}^\rho_{1,1}$, where $q\,=\,[0]$,
\begin{equation}\label{er1}
x([h]) \,=\, \Tr(h(\alpha)),\ y([h]) \,=\, \Tr(h(\beta)),\ z([h])\,=\, \Tr(h(\beta\alpha)),\,
\end{equation}
where $\alpha,\,\beta$ are the standard generators of $\pi_1(T^2\setminus\{o\},\,q)$,
represented by the curves
\begin{equation}\label{ea}
t\,\in\,[0,\,1]\,\longmapsto \,\alpha(t)\,=\, t\ \mod\ \Gamma
\end{equation}
and
\begin{equation}\label{eb}
t\,\in\,[0,
\,1]\,\longmapsto\, \beta(t)\,=\, t\sqrt{-1}\ \mod\ \Gamma\, ,
\end{equation}
respectively.

The variety ${\mathcal M}^\rho_{1,1}$ is defined by the equation
\begin{equation}\label{M11eq}
{\mathcal M}^\rho_{1,1}\,=\,\{(x,y,z)\,\in\,\C^3\,\mid \, x^2+y^2+z^2-xyz-2-2\cos(2\pi \rho)=0\}\, ;
\end{equation}
the details can be found in \cite{Gol}, \cite{Magn}.

\begin{lemma}\label{irr}
Take any $\rho\,\in\,]0,\,\tfrac{1}{2}[$, and consider a representation
\[h\,\colon\, \pi_1(T^2\setminus\{o\},\, q)\,\longrightarrow\, \SL(2,\C)\, ,\]
with $[h]\,\in\, {\mathcal M}^\rho_{1,1}$.
Then, the representation of the free group $F(s,t)$, with generators $s$ and $t$, defined by
\[s\,\longmapsto\, X\,:=\,h(\alpha)h(\alpha) \ \ \text{ and } \ \ t\,\longmapsto\,
Y\,:=\, h(\beta)h(\beta)\]
is reducible if and only if $x([h])y([h])\, =\, 0$,
where $x,\,y$ are the functions in \eqref{er1}.
\end{lemma}

\begin{proof}
It is known that, up to conjugation, we have
\begin{equation}\label{repxy}
h(\alpha)\,=\,\begin{pmatrix} x([h])&1\\-1&0\end{pmatrix}, \ \ h(\beta)
\,=\, \begin{pmatrix} 0&-\zeta\\ \zeta^{-1}& y([h])\end{pmatrix}\, ,
\end{equation}
where 
\begin{equation}\label{zeta1}
\zeta+\zeta^{-1}\,=\,z([h])
\end{equation}
(see \cite{Gol}). The lemma follows
from a direct computation by noting that a representation generated by two $\SL(2,\C)$ matrices
$A,\,B$ is reducible if and only if $AB-BA$ has a non-trivial kernel. If $AB-BA$ has a non-trivial kernel,
then $A$, $B$ and $AB-BA$ lie in a Borel subalgebra of ${\mathfrak s}{\mathfrak l}(2, \C)$.
\end{proof}

\subsection{Parabolic bundles and holomorphic connections}

\subsubsection{Parabolic bundles}\label{sec2.1}

We briefly recall the notion of a parabolic structure, mainly for the purpose of fixing the
notation. We are only 
concerned with the $\SL(2,\C)$--case, so our notation differs from the standard references, e.g., 
\cite{MSe,Biq,B}. Instead, up to a factor 2, we follow the notation of \cite{Pir}; see also \cite{HeHe} for this notation.

Let $V\,\longrightarrow\,\Sigma$ be a holomorphic vector bundle of rank two with trivial determinant bundle over a compact Riemann surface $\Sigma$.
Let $p_1,\, \cdots ,\, p_n\,\in\,\Sigma$ be pairwise distinct $n$ points, and set the divisor
\[D\,=\,p_1+\ldots +p_n\, .\]
For every $k\,\in\,\{1,\,\cdots ,\, n\}$, let $L_k\,\subset\, V_{p_k}$
be a line in the fiber of $V$ at $p_k,$ and also take
\[\rho_k\,\in\, ]0,\, \tfrac{1}{2}[\, .\]

\begin{definition}\label{def:par}
A {\it parabolic structure} on $V$ is given by the data
\[{\mathcal P}\,:=\, (D,\, \{L_1,\,\cdots ,\, L_n\},\, \{\rho_1,\,\cdots ,\, \rho_n\})\, ;\]
we call $\{L_k\}_{k=1}^n$ the quasiparabolic structure, and $\rho_k$ the parabolic weights.
A parabolic bundle over $\Sigma$ is given by a rank two holomorphic vector bundle $V$,
with $\bigwedge^2 V\,=\, {\mathcal O}_\Sigma$, together with a parabolic structure $\mathcal P$ on $V$.
\end{definition}

It should be emphasized that Definition \ref{def:par} is very specific to the case of
parabolic $\SL(2,\C)$--bundles. The parabolic degree of a holomorphic line subbundle
$F\,\subset\, V$ is
\[\text{par-deg}(F)\,:=\, {\rm degree}(F)+\sum_{k=1}^n \rho^F_k\, ,\]
where $\rho^F_k\,= \, \rho_k$ if $F_{p_k}\,=\,L_k$ and $\rho^F_k\,=\,
-\rho_k$ if $F_{p_k}\,\neq\, L_k$.

\begin{definition}
A parabolic bundle $(V,\,\mathcal P)$ is called {\it stable} if
$$\text{par-deg}(F)\, <\, 0$$
for every holomorphic line subbundle $F\,\subset \, V$.
\end{definition}

As before, ${\mathcal P}\,=\,(D\,=\, p_1+\ldots +p_n,\, \{L_1,\,\cdots ,\,L_n\},
\,\{\rho_1,\,\cdots ,\, \rho_n\})$ is a parabolic structure on a rank two holomorphic vector
bundle $V$ of trivial determinant.

A strongly parabolic Higgs field on a parabolic bundle $(V,\,{\mathcal P})$ is a holomorphic section
$$
\Theta\, \in\, H^0(\Sigma,\, \text{End}(V)\otimes K_\Sigma\otimes {\mathcal O}_\Sigma(D))
$$
such that $\text{trace}(\Theta) \,=\, 0$
and $L_k\,\subset\,\text{kernel}(\Theta(p_k))$ for all $1\, \leq\, k\, \leq\, n$.
These two conditions together imply that all the residues of a strongly parabolic Higgs field are nilpotent.

\subsubsection{Deligne extension}\label{Delext}

Take a flat $\SL(2,\C)$--connection $\nabla$ on a holomorphic
vector bundle $E_0$ over $T^2\setminus\{o\}$ (see \eqref{eo}), corresponding to a point
of ${\mathcal M}^\rho_{1,1}$. Then locally around $o\,\in\, T^2$, the connection $\nabla$ is
holomorphically $\SL(2,\C)$--gauge equivalent to the connection
\begin{equation}\label{local-normal-form-connection}
d+\dvector{\rho&0\\0&-\rho}\frac{dw}{w}
\end{equation}
on the trivial holomorphic bundle of rank two, where $w$ is a holomorphic coordinate function on
$T^2$ defined around $o$ with $w(o)\,=\, 0$. Take such a neighborhood $U_o$ of $o$ and
a holomorphic coordinate function $w$. Consider
the trivial holomorphic bundle 
$U_o\times {\mathbb C}^2\, \longrightarrow\, U_o$ equipped with the logarithmic
connection in \eqref{local-normal-form-connection}. Now glue the two holomorphic
vector bundles, namely $U_o\times {\mathbb C}^2$ and $E_0$, over
the common open subset $U_o\setminus\{o\}$ such that the
connection $\nabla\vert_{U_o\setminus\{o\}}$ is taken to the restriction of the logarithmic connection
in \eqref{local-normal-form-connection} to $U_o\setminus\{o\}$. This gluing is holomorphic
because it takes one holomorphic connection to another holomorphic connection. Consequently,
this gluing produces a holomorphic vector bundle
\begin{equation}\label{dV}
V\,\longrightarrow\, T^2
\end{equation}
of rank $2$. Furthermore, the connection $\nabla$ on $E_0\, \longrightarrow\, T^2\setminus\{o\}$
extends to a logarithmic connection on $V$ over $T^2$, because the
meromorphic connection in \eqref{local-normal-form-connection} is a logarithmic connection
on $U_o\times {\mathbb C}^2$. This resulting logarithmic connection on $V$ will
also be denoted by $\nabla$ (see \cite{De} for details).
The logarithmic connection on the exterior product
$$U_o\times \bigwedge\nolimits^2 {\mathbb C}^2\,=\, U_o\times {\mathbb C}\, \longrightarrow\, U_o$$
induced by the logarithmic connection on $U_o\times {\mathbb C}^2\, \longrightarrow\, U_o$
in \eqref{local-normal-form-connection} is actually a regular connection; in fact, it coincides with the trivial
connection on $U_o\times {\mathbb C}$. On the other hand, the connection on $\bigwedge^2E_0\,=\,
{\mathcal O}_{T^2\setminus\{o\}}$ induced by the connection
$\nabla$ on $E_0$ coincides with the trivial connection (recall that $\nabla$ is
a $\SL(2,\C)$--connection on $E_0$). Consequently,
\begin{enumerate}
\item $\bigwedge^2 V\, =\, {\mathcal O}_{T^2}$, where $V$ is the vector bundle in \eqref{dV}, and

\item the logarithmic connection on $\bigwedge^2 V$ induced by the logarithmic connection $\nabla$ on $V$
coincides with the trivial holomorphic connection on ${\mathcal O}_{T^2}$ induced by the de Rham differential $d$.
\end{enumerate} 
In particular, we have $\text{degree}(V)\,=\, 0$.

{}From Atiyah's classification of holomorphic vector bundles over any elliptic curve \cite{At}, the
possible types of the vector bundle $V$ in \eqref{dV} are:
\begin{enumerate}
\item $V\,=\, L\oplus L^{^*}$, with ${\rm degree}(L)\,=\,0$;

\item there is a spin bundle $S$ on $T^2$ (meaning a holomorphic line bundle of order two), such that
$V$ is a nontrivial extension
of $S$ by itself; and

\item $V\,=\, L\oplus L^{^*}$, with ${\rm degree}(L)\,>\,0$.
\end{enumerate}

\begin{lemma}\label{lem3.3}
Consider the vector bundle $V$ in \eqref{dV} for $\tfrac{1}{2}\,>\,\rho\,>\,0$. Then the
last one of the above three cases, as well as the special situation of the first
case where $L\,=\, S$ is a holomorphic line bundle with $S^{\otimes 2}\,=\, {\mathcal O}_{T^2}$, cannot occur.
\end{lemma}

\begin{proof}
First assume that case (3) occurs. So $V\,=\, L\oplus L^{^*}$, with ${\rm degree}(L)\,>\,0$.
We have
$$
\text{degree}(\text{Hom}(L,\, L^{^*})\otimes K_{T^2}\otimes {\mathcal O}_{T^2}(o))\,=\,
1-2\text{degree}(L) \, <\, 0\, ,
$$
where $K_{T^2}$ is the holomorphic cotangent bundle.
So the second fundamental form of $L$ for $\nabla$, which is a holomorphic section of
$\text{Hom}(L,\, L^{^*})\otimes K_{T^2}\otimes {\mathcal O}_{T^2}(o)$, vanishes identically.
Consequently, the logarithmic connection $\nabla$ on $V$ preserves the line subbundle $L$.
Since $L$ admits a logarithmic connection, with residue $r\, \in\, \{\rho,\, -\rho\}$ at $o$, we have
$$
{\rm degree}(L)+ r\,=\, 0
$$
\cite[p.~16, Theorem 3]{Oh}. But this contradicts the fact that
$\rho\,\in\,]0,\, \tfrac{1}{2}[$. So case (3) does not occur.

Next assume that $V\,=\, S\oplus S^{^*}\,=\, S\otimes_{\mathbb C} {\mathbb C}^2$, where
$S$ is a holomorphic line bundle with $S^{\otimes 2}\,=\, {\mathcal O}_{T^2}$. Then $V$ admits a
holomorphic connection, and moreover
$$
H^0(T^2,\, \text{End}(V)\otimes K_{T^2}\otimes {\mathcal O}_{T^2}(o))\,=\,
H^0(T^2,\, \text{End}(V)\otimes K_{T^2})\,=\, \text{End}({\mathbb C}^2)\, .
$$
So $V$ does not admit a logarithmic connection singular exactly over $o\, \in\, T^2$, because
such logarithmic connections form an affine space for the vector space
$H^0(T^2,\, \text{End}(V)\otimes K_{T^2}\otimes {\mathcal O}_{T^2}(o))$.
\end{proof}

\subsubsection{Parabolic structure from a logarithmic connection}\label{sec3.e}

Consider a logarithmic connection $\nabla$ on a holomorphic bundle $V$ of rank two
and with trivial determinant over a compact Riemann surface $\Sigma$.
We assume that $\nabla$ is a $\SL(2,\C)$--connection, meaning
the logarithmic connection on $\bigwedge^2 V$ induced by
$\nabla$ is a holomorphic connection with trivial monodromy; note that
this implies that $\bigwedge^2 V\,=\, {\mathcal O}_\Sigma$. Let $p_1,\,\cdots ,\,p_n\,\in\,\Sigma$ be the singular
points of $\nabla$.
We also assume that the residue \[\text{Res}_{p_k}(\nabla)\,\in\,\text{End}_0(V_{p_k})\]
of the connection $\nabla$ at every point $p_k$, $1\,\leq\, k\, \leq\, n$, has two real eigenvalues
$\pm\rho_k$, with $\rho_k\,\in\,]0,\, \tfrac{1}{2}[.$ For every $1\, \leq\, k\, \leq\, n$, let 
\[L_k\,:=\, \text{Eig}(\text{Res}_{p_k}(\nabla),\, \rho_k)\, \subset\, V_{p_k}\]
be the eigenline of the residue of $\nabla$ at $p_k$ for the eigenvalue $\rho_k$.

The logarithmic connection $\nabla$ gives rise 
to the parabolic structure \[{\mathcal P}\,=\,(D=p_1+\ldots +p_n,\, \{L_1,\, \cdots,\, L_n\},\,
\{\rho_1,\,\cdots ,\, \rho_n\})\, .\]
It is straightforward to check that another such logarithmic connections $\nabla^1$ on $V$
induces the same parabolic structure $P$ if and only if $\nabla -\nabla^1$
is a strongly parabolic Higgs field on $(V,\, {\mathcal P})$, in the sense of Section 
\ref{sec2.1}.

It should be mentioned that in \cite{MSe}
a different local form
of the connection is used (instead of the local form in \eqref{local-normal-form-connection}). In that
case the Deligne extension gives a rank two holomorphic vector bundle $W$ (instead of $V$) with
$\bigwedge^2 W\,=\, {\mathcal O}_{\Sigma}(-D)$ (instead of $\bigwedge^2 V\,=\, {\mathcal O}_{\Sigma}$),
while the parabolic weights at $p_k$ become $\rho_k,\, 1-\rho_k$ (instead of $\rho_k,\, -\rho_k$).

A theorem of Mehta and Seshadri \cite[p.~226, Theorem 4.1(2)]{MSe}, and Biquard \cite[p.~246, Th\'eor\`eme 
2.5]{Biq} says that the above construction of a parabolic bundle $(V,\, {\mathcal P})$ from a logarithmic 
connection $\nabla$ produces a bijection between the stable parabolic bundles (in the sense of Section 
\ref{sec2.1}) on $(\Sigma,\, D)$ and the space of isomorphism classes of irreducible flat ${\rm 
SU}(2)$--connections on the complement $\Sigma\setminus D$. See, for example, \cite[Theorem 3.2.2]{Pir} for 
our specific situation. As a consequence of the above theorem of \cite{MSe} and \cite{Biq}, for every logarithmic 
connection $\nabla$ on $V$ which produces a stable parabolic structure $\mathcal P$, there exists a unique 
strongly parabolic Higgs field $\Theta$ on $(V,\, {\mathcal P})$ such that the holonomy of the flat 
connection $\nabla+\Theta$ is contained in ${\rm SU}(2)$. This flat ${\rm SU}(2)$--connection 
$\nabla+\Theta$ is irreducible, because the parabolic bundle is stable.

\subsection{Abelianization}

Take $o\, \in\, T^2$ as in \eqref{eo}.
In \cite{He3}, representatives $\nabla$ for each gauge class in
$\mathcal M_{1,1}^\rho$ are computed for the special case where $\rho\,=\,\tfrac{1}{6}$ and
$L\,\in \,{\rm Jac}(T^2)\setminus\{S \,\mid\, S^{\otimes 2}\,=\, K_{T^2}\}$. 
We shall show (see Proposition \ref{explicit_coeff}) that for general $\rho$ and
$$L\,\in \,{\rm Jac}(T^2)\setminus\{S \,\mid\, S^{\otimes 2}\,=\, {\mathcal O}_{T^2}\}\, ,$$ the corresponding connection
$\nabla$ is of the form
\begin{equation}\label{abel-connection}
\nabla\,=\,\nabla^{a,\chi,\rho}\,=\,\dvector{\nabla^L &\gamma^+_\chi\\ \gamma^-_\chi &
\nabla^{L^*} }\, ,
\end{equation}
where $a,\, \chi\,\in\,\C$, and
$$\nabla^L\,=\,d+a\cdot dw+\chi\cdot d\overline{w}$$
is a holomorphic connection on $L$ with $\nabla^{L^*}$ being the dual connection on $L^{^*}$;
here $w$ denotes a complex affine coordinate on $T^2\,=\, {\mathbb C}/\Gamma$. The
off--diagonal terms in \eqref{abel-connection} can be described explicitly in terms of the
theta functions as explained below.

Before doing so, we briefly describe both the Jacobian and the rank one de Rham moduli space for $T^2$ in
terms of some useful coordinates. Let $d\,=\,\partial+\overline\partial$
be the decomposition of the de Rham differential $d$ on $T^2$ into its $(1,0)$--part $\partial$ and
$(0,1)$--part $\overline\partial$. It
is well--known that every holomorphic line bundle of degree zero on $T^2$ is given by a
Dolbeault operator
\[\overline{\partial}^\chi\,=\,\overline{\partial} +\chi\cdot d\overline{w}\]
on the $C^\infty$ trivial line bundle $T^2\times {\mathbb C}\,\longrightarrow\,
T^2$, for some $\chi\,\in\,\C$, where 
$w$ is an affine coordinate function on $\C/(\Z+\sqrt{-1}\Z)\,=\,T^2$ (note that 
$d\overline{w}$ does not depend on the choice of the affine function $w$). So the operator
$\overline{\partial}^\chi$ sends a locally defined $C^\infty$ function $f$ to the $(0,\,1)$-form
$\overline{\partial}f +\chi f\cdot d\overline{w}\,=\,\overline{\partial}f +\chi f\cdot \overline{\partial w}$.
Two such differential operators
\[\overline{\partial}^{\chi_1}\ \ \text{ and } \ \ \overline{\partial}^{\chi_2}\]
determine isomorphic holomorphic line bundles if and only if $\overline{\partial}^{\chi_1}$ and 
$\overline{\partial}^{\chi_2}$ are gauge equivalent. They are gauge equivalent if and only 
if
\begin{equation}\label{gauge-lattice}\chi_2-\chi_1\,\in\, \Gamma^*:=\,\pi\Z+\pi \sqrt{-1}\Z\end{equation}

Similarly, flat line bundles over $T^2$ are given by the connection operator
\[d^{a,\chi}\,=\,d+a\cdot dw+\chi\cdot d\overline{w}\]
on the $C^\infty$ trivial line bundle $T^2\times {\mathbb C}\,\longrightarrow\,
T^2$, for some $a,\, \chi\,\in\,\C$. Moreover two connections $d^{a_1,\chi_1}$ and $d^{a_2,\chi_2}$
are isomorphic if and only if
\[(a_2-a_1) + (\chi_2-\chi_1) \,\in\, 2\pi \sqrt{-1} \Z \ \ \text{ and }\ \
(a_2-a_1) - (\chi_2-\chi_1)\,\in\, 2\pi \sqrt{-1} \Z\, .\]

The (shifted) theta function for
$\Gamma\,=\,\Z+\Z\sqrt{-1}$ will be denoted by
$\vartheta$. In other words, $\vartheta$ is 
the unique (up to a multiplicative constant) entire function satisfying $\vartheta(0) \,=\, 0$ and
\[\vartheta(w+ 1) \,=\, \vartheta (w),\,\, \vartheta(w+\sqrt{-1}) \,=\, - \vartheta (w)\exp(-2\pi
\sqrt{-1}w)\, .\]
Then the function \[t_{x}(w-\tfrac{1+\sqrt{-1}}{2}) \,:=\, \frac{\vartheta(w-x)}{\vartheta(w)}\exp(-\pi x(w-\overline{w}))\]
is doubly periodic on $\C\setminus(\tfrac{1+\sqrt{-1}}{2}+\Gamma)$ with respect to $\Gamma$
and satisfies the equation
\[(\dbar-\pi xd\overline{w})t_{x}\,=\,0\, .\]
Thus $t_x$ is a meromorphic section of the holomorphic
bundle $L(\overline{\partial}^{-\pi x})\, :=\,[\overline{\partial}^{-\pi x}]$
(it is the holomorphic line bundle given by the Dolbeault operator $\dbar-\pi xd\overline{w}$). Notice that
for $x\,\notin\,\Gamma$, the section $t_x$ 
has a simple zero at $w\,=\,x$ and a first order pole at $w \,=\, o$ (see \eqref{eo}).
 Moreover, up to scaling by a complex number, this $t_x$ is the unique meromorphic section of
$L(\overline{\partial}^{-\pi x})\, :=\, [\overline{\partial}^{-\pi x}]$ with a simple pole at $o$.

For $\frac{1}{2}\, >\, \rho\, >\, 0$, if $V$ in \eqref{dV} is of the form
$V\,=\, L\oplus L^{^*}$, then
from Lemma \ref{lem3.3} it follows that $\text{degree}(L)\,=\, 0$ and $L$ is
not a spin bundle. In other words,
\[L\,=\, L(\overline{\partial}+\chi\cdot d\overline{w})\]
for some $\chi\,\in\,\mathbb C\setminus\tfrac{1}{2}\Gamma^*$; compare with \eqref{gauge-lattice}.

\begin{proposition}\label{explicit_coeff}
For any $\rho\,\in\, [0,\, \tfrac{1}{2}[$,
take $[\nabla]\,\in\, {\mathcal M}_{1,1}^\rho$ such that its
Deligne extension is given by the holomorphic vector bundle
$V\,=\, L\oplus L^*$, where $L\,=\, L(\overline{\partial}+\chi d\overline{w})$
is a holomorphic line bundle on $T^2$ of degree zero such that $L^{\otimes 2}\, \not=\, {\mathcal O}_{T^2}$.
Set $x\,=\,-\frac{1}{\pi }\chi$, so $x\,\notin\,\tfrac{1}{2}\Gamma.$
Then, there exists $$a\,\in\,\C$$ such that one representative of $[\nabla]$ is given by
\[\nabla^{a,\chi,\rho}\] as in \eqref{abel-connection}, where
the second fundamental forms $\gamma^+_\chi$ and $\gamma^-_\chi$ in \eqref{abel-connection} are given
by the meromorphic $1$--forms
\begin{equation}\label{gammamp}
\gamma^+_\chi([w])\,=\,\rho \tfrac{\vartheta'(0)}{\vartheta(-2x)}t_{2x}(w)dw\ \
\text{ and }\ \
\gamma^-_\chi([w])\,=\,\rho \tfrac{\vartheta'(0)}{\vartheta(2x)}t_{-2x}(w)dw
\end{equation}
with values in the holomorphic line bundles
$L(\dbar+ 2\chi d\overline{w})$ and $L(\dbar- 2\chi d\overline{w})$ respectively.
\end{proposition}

\begin{proof}
Using Section \ref{Delext} 
we know that there exists a representative $\nabla$ of $[\nabla]$ such that its
$(0,1)$--part ${\overline\partial}^\nabla$ is given by
\[{\overline\partial}^\nabla\,=\,{\overline\partial}+\begin{pmatrix} \chi d\overline{w}&0
\\ 0& -\chi d\overline{w}\end{pmatrix}\, .
\]
The $(1,0)$--part $\partial^\nabla$ is given by
 $\partial^\nabla\,=\,\partial+\Psi$, where $\Psi\,=\,\begin{pmatrix} A & B\\ C& -A \end{pmatrix}$
is an $\text{End}(V)$--valued meromorphic $1$--form
on $T^2$, with respect to the holomorphic structure ${\overline\partial}^\nabla$,
such that $\Psi$ has a simple pole at $o$ and $\Psi$ is holomorphic elsewhere. In particular, $A$ is a
meromorphic $1$--form on $T^2$ with simple pole at $o$ (see \eqref{eo}), and hence by the residue
theorem it is in fact holomorphic, i.e.,
\[A\,=\, adw\]
for some $a\,\in\,\C$. 
Furthermore, $B$ and $C$ are meromorphic $1$--forms with values in the holomorphic bundles
$L(\dbar +2\chi d\overline{w})$ and $L(\dbar -2\chi d\overline{w})$, respectively.
Note that for $$x\,=\, -\frac{1}{\pi }\chi \,\in\,\tfrac{1}{2}\Gamma\, ,$$ the holomorphic line bundle
$L(\dbar +2\chi d\overline{w})$ would be the trivial and
$B$ and $C$ cannot have non-trivial residues at $o$ by the residue theorem. 
The determinant of the residue of $\Psi$ at $o$ is $-\rho^2$
by \eqref{local-normal-form-connection}. Therefore, from the holomorphicity
of $A$ we conclude that the quadratic residue of the meromorphic quadratic differential $BC$ is
\[\text{qres}_o(BC)\,=\,\rho^2\, .\]
From the discussion prior to this proposition it follows
that there is a unique meromorphic section 
of $L(\dbar\pm2\chi d\overline{w})$ with a simple pole at $o$.
Thus, after a possible constant diagonal gauge transformation, from the uniqueness, up to
scaling, of the meromorphic section of $L(\dbar\pm2\chi d\overline{w})$ with simple pole at
$o$, it follows that
\[ B\,=\,\gamma^+_\chi \ \ \text{ and } \ \ C\,=\,\gamma^-_\chi\, ,\]
where $\gamma^+_\chi$ and $\gamma^-_\chi$ are the second fundamental forms 
\eqref{abel-connection}.
This completes the proof.
\end{proof}

\begin{proposition}\label{Pro-stab}
Assume that $\rho\,\in\, ]0,\, \tfrac{1}{2}[$. Take $[\nabla]\, \in\, {\mathcal M}^\rho_{1,1}$
such that the corresponding bundle $V$ in \eqref{dV} is of the form $L\oplus L^{^*}$
(so $L$ is not a spin bundle but its degree is zero by Lemma \ref{lem3.3}).
Then, the rank two parabolic bundle corresponding to $[\nabla]$ (see Section
\ref{sec3.e}) is parabolic stable.
\end{proposition}

\begin{proof}
The two holomorphic line bundles $L$ and $L^{^*}$ are not isomorphic, because $L$ is not a spin 
bundle. From this it can be shown that any holomorphic subbundle of degree zero
$$
\xi\, \subset\, V\,=\, L\oplus L^{^*}
$$
is either $L$ or $L^{^*}$. Indeed, if $\xi$ is a degree zero holomorphic line bundle different from
both $L$ and $L^{^*}$, then
$$
H^0(T^2,\, \text{Hom}(\xi,\, V))\,=\, H^0(T^2,\, \text{Hom}(\xi,\, L))\oplus
H^0(T^2,\, \text{Hom}(\xi,\, L^{^*}))\,=\, 0\, .
$$

As the residue in \eqref{abel-connection} is off--diagonal (with respect to the holomorphic decomposition 
$V\,=\,L\oplus L^*$), the above observation implies that
every holomorphic line subbundle $\xi\, \subset\, V$ of degree zero has parabolic degree $-\rho$. 
On the other hand, the parabolic degree of a holomorphic line subbundle of negative degree
is less than or equal to 
$-1+\rho\,<\,0\,.$
Consequently, the parabolic bundle corresponding to $[\nabla]$ is stable.
\end{proof}

\section{Flat connections on the 4-punctured torus}\label{sec:4-pun}

Consider $$\widehat{T}^2\,:=\,\C/(2\Z+2\sqrt{-1}\Z)$$
and the 4--fold covering
\begin{equation}\label{mPi}
\Pi\,\colon\, \widehat{T}^2\,\longrightarrow\, T^2\,:=\,\C/(\Z+\sqrt{-1}\Z)
\end{equation}
produced by the identity map of $\mathbb C$. Let
$$\{p_1,\,p_2,\, p_3,\, p_4\}\,:=\, \Pi^{-1}(o) \, \subset\, \widehat{T}$$
be the preimage of $o\,=\,\left[\tfrac{1+\sqrt{-1}}{2}\right]\, \in\, T^2$ (see \eqref{eo}). Fix
\[\rho\,=\, 0\, ,\]
and consider the corresponding connection $\nabla\,=\,\nabla^{a,\chi,0}$ in \eqref{abel-connection}.
We use $\Pi$ in \eqref{mPi}
to pull back this connection to $\widehat{T}^2$.

Let $$h\, :\, \pi_1(\widehat{T}^2,\, q)\,\longrightarrow \, \text{SL}(2,{\mathbb C})$$
be the monodromy representation for $\Pi^*\nabla^{a,\chi,0}$, where $q\,=\,[0] \in \widehat{T}^2$.

The traces $$T_1(\chi,a)\,=\, \text{tr}(h(\widehat{\alpha}))\ \ \text{ and }\ \
T_2(\chi,a)\,=\, \text{tr}(h(\widehat{\beta}))$$
along
\begin{equation}\label{tab}
\widehat{\alpha}\, ,\,\,\widehat{\beta}\,\in\, \pi_1(\widehat{T}^2\setminus\{p_1,\cdots ,p_4\},\, q)\ \ 
\end{equation}
with representatives
\[\widehat{\alpha}\,\colon\, [0,\, 2]\,\longmapsto\, 2t\ \ \mod\ 2\Gamma \,\in\, \widehat{T}^2 \]
and
\[\widehat{\beta}\,\colon\, [0,\,2]\,\longmapsto\, 2t\sqrt{-1}\ \ \mod\ 2\Gamma \,\in\, \widehat{T}^2 \]
(see Figure \ref{figure1}, Figure \ref{figure2})
are given by
\[T_1(\chi,\,a)\,=\,\exp(-2(a+\chi))+\exp(2(a+\chi))\]
and
\[T_2(\chi,\,a)\,=\,\exp(2\sqrt{-1}(-a+\chi))+\exp(2\sqrt{-1}(a -\chi))\]
respectively, while the local monodromy of $\Pi^*\nabla^{a,\chi, 0}$
around each of $p_1,\,\cdots ,\,p_4$ is trivial, because we have $\rho\,=\,0$.

In the following, set 
\begin{equation}\label{chifix}
\chi\,=\,\frac{\pi}{4}(1-\sqrt{-1})\quad \text{and} \quad a_k\,=\,-\frac{\pi}{4}(1+\sqrt{-1})+k\pi(1+\sqrt{-1}) 
\end{equation}
for all $k\,\in\,\Z.$ Then we have
\begin{equation}\label{T1eq}T_1(\chi,a_k)\,=\,-(\exp(-2k\pi)+\exp(2k\pi))\,\in\, \R\end{equation}
\begin{equation}\label{T2eq}T_2(\chi,a_k)\,=\,-(\exp(-2k\pi)+\exp(2k\pi))\,\in\, \R\, ;\end{equation}
as before, $T_1(\chi,a_k)$ and $T_2(\chi,a_k)$ are the traces of holonomies of
$\Pi^*\nabla^{a_k,\chi,0}$ (see \eqref{abel-connection} and \eqref{mPi}) along $\widehat\alpha$ and
$\widehat\beta$ respectively
(see \eqref{tab}). Moreover, a direct computation shows that
\begin{equation}\label{der0}
(s,\,t)\,\longmapsto\, (\Im(T_1(\chi,\,a_k+s+\sqrt{-1}t)
),\, \Im(T_2(\chi,\,a_k+s+\sqrt{-1}t)
))\end{equation}
is a local diffeomorphism at $(s,\,t)\,=\,(0,\,0)$ by the implicit function theorem.

\begin{theorem}\label{real-mon-hatT2}
Let $k\,\in\,\Z\setminus\{0\}$, $\chi\,=\,\frac{\pi}{4}(1-\sqrt{-1})$ and $a_k
\,=\,-\frac{\pi}{4}(1+\sqrt{-1})+k\pi(1+\sqrt{-1})$. Then
there exists $\epsilon\,>\,0$ such that for each $\rho\,\in\,]0,\,\epsilon[$,
there is a unique number $a\,\in\,\C$ near $a_k$ satisfying the condition that the monodromy of
the flat connection
\[
\Pi^*\nabla^{a,\chi,\rho}
\] on $\widehat{T}^2\setminus\{p_1,\cdots ,p_4\}$ (see \eqref{abel-connection} and \eqref{mPi})
is irreducible and furthermore the image of the monodromy homomorphism is conjugate to a subgroup of $\SL(2,\R)$.
\end{theorem}

\begin{proof}
Using the fact that the map in \eqref{der0} is a local diffeomorphism, there exists for each sufficiently small
$\rho$ a unique complex number $a$ such that the traces $T_1$ and $T_2$,
of holonomies of $\nabla^{a,\chi,\rho}$ along $\widehat\alpha$ and
$\widehat\beta$ respectively (see \eqref{tab}), are real.
Because $k\,\neq\, 0$, and $\rho$ is small, we obtain from
\eqref{T1eq} and \eqref{T2eq} that these traces satisfy 
\[T_1\,<\,-2\ \ \text{ and }\ \ T_2\,<\,-2\, .\]

Recall the general formula 
\begin{equation}\label{tXY}\text{tr}(X)\text{tr}(Y)\,=\,\text{tr}(XY)+\text{tr}(XY^{-1})\end{equation}
for $X,\,Y\,\in\, \text{SL}(2,\C)$.
Let
\begin{equation}\label{exy}
x\,=\,\text{tr}(h(\alpha))\ \ \text{ and } \ \ y\,=\,\text{tr}(h(\beta))
\end{equation}
be the traces of the monodromy homomorphism $h$ of the connection $\nabla^{a,\chi,\rho}$
on $T^2\setminus\{0\}$ along $\alpha$ and $\beta$ defined in \eqref{ea} and \eqref{eb} respectively.

Applying \eqref{tXY} to
\[X\,=\, h(\alpha)\,=\, Y \ \ \text{( respectively, }\ \ X\,=\, h(\beta)\,=\,Y)\]
we obtain that $x$ (respectively, $y$) in \eqref{exy} must be purely imaginary.
Then it can be checked directly that the trace along any closed curve in the 4--punctured
torus is real. The fact that
\[z\,=\,\text{tr}(h(\alpha\circ\beta))\]
is real is a direct consequence of \eqref{M11eq} combined with the above observation
that $x,\,y\,\in\,\sqrt{-1}\R$. Using \eqref{tXY} repeatedly (compare with \cite{Gol}) it is deduced
that the trace of the monodromy along any closed curve on $\widehat{T}^2$ is real.

For $\rho\,\neq\,0$ sufficiently small, the connection
$\Pi^*\nabla^{a,\chi,\rho}$ on $\widehat{T}^2$ is irreducible as a
consequence of Lemma \ref{irr} --- note that the condition $xy\,\neq\, 0$ follows directly from
the fact that $\rho\,\neq\, 0$ --- applied to $h(\widehat{\alpha})$ and $h(\widehat{\beta})$
(see \eqref{tab}).

We will show that the image of the monodromy homomorphism $h$ is conjugate to a subgroup
of $\text{SL}(2,\R)$.

To prove the above statement, first note that since the monodromy $h$ is irreducible and has all traces real,
the homomorphism $h$ is in fact
conjugate to its complex conjugate representation $\overline h$, meaning there exists $C\,\in\,
\SL(2,\C)$ such that
\[C^{-1}\overline{h} C\,=\, h\, .\]
Applying this equation twice we get that
\[\overline{C}C\,=\, \pm \text{Id}\]
because $h$ is irreducible.
If $$\overline{C}C\,=\, -\text{Id}\, ,$$
a straightforward computation shows that $h$ is conjugate to a unitary representation.
Since the traces of some elements in the image of the monodromy are not contained in $[-2,\, 2]$, we
are led to a contradiction.

Thus, we have $$\overline{C}C\,=\, \text{Id}\, .$$
A direct computation gives that
\[C\,=\,\overline{D}^{-1} D\]
for some $D\,\in\, \SL(2,\C)$. Consequently, we have
\[ DhD^{-1}\,=\, \overline{D}\overline{h} \overline{D}^{-1}\, .\]
Hence the image of the monodromy homomorphism $h$ is conjugate to a subgroup of $\text{SL}(2,\R)$.
\end{proof}

We shall use the following theorem.

\begin{theorem}\label{thetrivcon}
Let $\chi\,=\,\frac{\pi}{4}(1-\sqrt{-1}).$
For every $\rho\,\in\, [0,\, \tfrac{1}{2}[$, there exists
$a^u\,\in\, \C$ such that
\[\Pi^*\nabla^{a^u,\chi,\rho}\]
is a reducible unitary connection satisfying the following condition:
the monodromies of $\Pi^*\nabla^{a^u,\chi,\rho}$ along
$$\widehat\alpha \,\in\, \pi_1(\widehat{T}^2\setminus\{p_1,\cdots ,p_4\},\, q)
\ \ and\ \ \widehat\beta \,\in\,
\pi_1(\widehat{T}^2\setminus\{p_1,\cdots ,p_4\},\, q)
$$
(see \eqref{tab}) are both $-{\rm Id}$. Moreover, the monodromies around the points $p_1,\,\cdots,\,p_4$ are (after simultaneous conjugation) given by
\[\begin{pmatrix} \exp(2\pi\sqrt{-1}\rho)&0\\0&\exp(-2\pi\sqrt{-1}\rho)\end{pmatrix},\, \begin{pmatrix} \exp(-2\pi
\sqrt{-1}\rho)&0\\0&\exp(2\pi\sqrt{-1}\rho)\end{pmatrix}\, ,
\]
\[
\begin{pmatrix} \exp(2\pi\sqrt{-1}\rho)&0\\0&exp(-2\pi\sqrt{-1}\rho)\end{pmatrix}, \,
\begin{pmatrix} \exp(-2\pi\sqrt{-1}\rho)&0\\0&\exp(2\pi\sqrt{-1}\rho)\end{pmatrix}\]
respectively.
\end{theorem}

\begin{proof}
First, for any $\rho\,\in\, ]0,\, \tfrac{1}{2}[$ and $a\,\in\,\C$, the parabolic bundle on $T^2$ determined
by $\nabla^{a,\chi,\rho}$ with $\chi\,=\,\frac{\pi}{4}(1-\sqrt{-1})$ is
stable (see Proposition \ref{Pro-stab}); this stable parabolic bundle on $T^2$ will be denoted by $W_*$.
Note that all the strongly parabolic Higgs fields on this parabolic bundle are given
by constant multiples scalar of
\[\begin{pmatrix} dw & 0\\ 0&-dw\end{pmatrix}\, .\]
In view of the theorem of Mehta--Seshadri and Biquard (\cite{MSe}, \cite{Biq})
mentioned in Section \ref{sec3.e}, there exists $a^u\,\in\,\C$ such that
\[\nabla^{a^u,\chi,\rho}\] has unitary monodromy on $T^2.$ Then, the flat connection
$\Pi^*\nabla^{ a^u,\chi,\rho}$ on $\widehat{T}^2$ has unitary monodromy as well,
where $\Pi$ is the projection in \eqref{mPi}.

On the other hand, the pulled back parabolic bundle $\Pi^*W_*$ on $\widehat{T}^2$
is strictly semi-stable, because
$\chi\,=\,\frac{\pi}{4}(1-\sqrt{-1})$ and $\widehat{T}^2\,=\,\C/(2\Gamma)$
for the specific lattice $2\Gamma\,=\,2\Z+2\sqrt{-1}\Z$
(it can be proved by a direct computation, but it also follows from \cite[Theorem 3.5
and Section 2.4]{HeHe}), so that the unitary connection $\Pi^*\nabla^{ a^u,\chi,\rho}$ is automatically reducible.

In order to compute the entire monodromy representation, set
$x \,=\, y \,= \,0$, and consider the unique positive solution of $z$ in \eqref{M11eq}. Note, that 
if $\rho\,=\,0,$ then $z\,=\,2$ and $a^u\,=\,-\overline\chi$, with $\chi$ given by 
\eqref{chifix}. Now for a general real $\rho$, using \eqref{repxy} after setting $\zeta\,=\, \exp(\pi\sqrt{-1}\rho)$
there, we see that the representation $h$ of the 
fundamental group of the $1$--punctured torus given by $x(h)\,=\, 0\, =\, y(h)$ and $z(h)\,=\, 
z$ induces a unitary reducible representation of the fundamental group of the 4--punctured 
torus. 

To identify the representation $h$ with the monodromy representation of $\nabla^{ a^u,\chi,\rho}$,
we note that, for $\rho\,<\,\tfrac{1}{4}$ (it suffices to consider this
case for our proof), it can be shown that the parabolic structure on the holomorphic vector bundle
\begin{equation}\label{epb}
L\oplus L^*\,\longrightarrow\, \widehat{T}^2
\end{equation}
cannot be strictly semi-stable if $L^{\otimes 2}$ is not trivial. 
Indeed, the lines giving the quasiparabolic structure are not contained in $L$ or 
$L^*$ by \eqref{abel-connection}. On the other hand, these two subbundles, namely $L$ and $L^{^*}$, are the only 
holomorphic subbundles of degree zero; this follows from the assumption that $L^{\otimes 2}\,\neq\, {\mathcal 
O}_{\widehat{T}^2}$, because $H^0(\widehat{T}^2,\, \text{Hom}(\xi,\, L\oplus L^*))\,=\, 0$, if
$\xi$ is a holomorphic line bundle of degree zero which is different from both $L$ and $L^*$.
Hence the parabolic structure on the holomorphic vector bundle in \eqref{epb}
cannot be strictly semi-stable if $L^{\otimes 2}\, \not=\, {\mathcal O}_{\widehat{T}^2}$.

By continuity of the monodromy representation of 
$\Pi^*\nabla^{a^u,\chi,\rho}$ with respect to the parameters $(a^u,\,\chi,\,\rho)$, the 
representation of $\nabla^{ a^u,\chi,\rho}$ must be the unitary reducible representation 
$h$ with $x(h)\,=\,0\,=\,y(h)$ and positive $z(h)\,=\,z$.
 
Finally, the corresponding monodromies of $\Pi^*\nabla^{ a^u,\chi,\rho}$ can be computed 
using \eqref{repxy}, where $\zeta\,=\, \exp(\pi\sqrt{-1}\rho):$
the monodromies along $\widehat\alpha$ and 
$\widehat\beta$ (see \eqref{tab}) are given by $h(\alpha) h(\alpha)$ and $h(\beta) h(\beta)$
respectively, and both are equal 
to $-\text{Id}$ by \eqref{repxy}, and the monodromies (based at $q\,=\,[0] $) around $p_1,\,\cdots,\,p_4$ are given by
\begin{equation}
\begin{split}
&h(\beta)^{-1}h(\alpha)^{-1}h(\beta)h(\alpha)\,=\,\begin{pmatrix} \exp(2\pi\sqrt{-1}\rho)&0\\0&\exp(-2\pi\sqrt{-1}\rho)\end{pmatrix},\\
&h(\alpha)h(\beta)^{-1}h(\alpha)^{-1}h(\beta)\,=\,\begin{pmatrix} \exp(-2\pi\sqrt{-1}\rho)&0\\0&\exp(2\pi\sqrt{-1}\rho)\end{pmatrix},\\
&h(\beta)h(\alpha)h(\beta)^{-1}h(\alpha)^{-1}\,=\,\begin{pmatrix} \exp(2\pi\sqrt{-1}\rho)&0\\0&\exp(-2\pi\sqrt{-1}\rho)\end{pmatrix},\\
&h(\alpha)^{-1}h(\beta)h(\alpha)h(\beta)^{-1}\,=\,\begin{pmatrix} \exp(-2\pi\sqrt{-1}\rho)&0\\0&\exp(2\pi\sqrt{-1}\rho)\end{pmatrix}
\end{split}
\end{equation}
respectively; compare with Figure \ref{figure1}, Figure \ref{figure2}. \end{proof}

\section{Flat irreducible $\SL(2,\R)$--connections on compact surfaces}\label{sec:maiN}

We assume that 
\[\rho\,=\,\frac{1}{2p}\, ,\]
for some $p\,\in\,\N$ odd, with $\rho$ being small enough so that Theorem \ref{real-mon-hatT2}
is applicable.

The torus $\widehat{T}^2$ in \eqref{mPi} is of square conformal type, and it is given by the algebraic equation
\begin{equation}\label{widehattyz}y^2\,=\, \frac{z^2-1}{z^2+1}\, .\end{equation}
Without loss of any generality, we can assume that the four points \[\{p_1,\,\cdots ,\,p_4\}
\,=\, \Pi^{-1}(\{o\})\, ,\]
where $\Pi$ is the map in \eqref{mPi} and $o$ is the point in \eqref{eo},
are the branch points of the function $z$. With the labelling of the points as in Figure \ref{figure1},
Figure \ref{figure2}, i.e.,
\[p_1\,=\,\left[\frac{1+\sqrt{-1}}{2}\right],\ p_2\,=\,\left[\frac{3+\sqrt{-1}}{2}\right],\
p_3\,=\,\left[\frac{3+3\sqrt{-1}}{2}\right],
\]
\[
p_4\,=\,\left[\frac{1+3\sqrt{-1}}{2}\right]\,\in\,\C/(2\Z+2\sqrt{-1}\Z)\,=\, \widehat{T}^2\, ,
\]
 it can be shown (for example using the Weierstrass $\wp$-function) that the $(y,\, z)$ coordinates of $p_1,\,\cdots ,\, p_4$ and $q\,=\,[0]$ can be 
 chosen to be
\begin{equation}\label{yz-coord}p_1\,=\,(0,\,1), \ \ p_2\,=\,(\infty,\,\sqrt{-1}),\ \ p_3\,=\,(0,\,-1), \ \ 
p_4\,=\,(\infty,\,-\sqrt{-1}) , \ \ 
q\,=\,(\sqrt{-1},\,0)\, .\end{equation}

Define the compact Riemann surface $\Sigma$ by the algebraic equation
\begin{equation}\label{sigmayz}
x^{2p}\,=\,\frac{z^2-1}{z^2+1}\, .
\end{equation}
Consider the $p$--fold covering
\begin{equation}\label{zp}
\Phi_p\,\colon\, \Sigma\,\longrightarrow\, \widehat{T}^2\, ,\ \ (x,\, z)
\, \longmapsto \, (x^p,\, z)\, ,
\end{equation}
which is totally branched over $p_1,\, \cdots,\, p_4$. Denote the inverse image
$\Phi^{-1}_p(p_i)$, $1\, \leq\, i\, \leq\, 4$, by $P_i$ (see Figure \ref{figure3}).

For a connection $\nabla^A$ (respectively, $\nabla^B$) on a vector bundle $A$
(respectively, $B$), the induced connection $(\nabla^A\otimes\text{Id}_B)\oplus
(\text{Id}_A\otimes\nabla^B)$ on $A\otimes B$ will be denoted by $\nabla^A\otimes\nabla^B$
for notational convenience.

There are holomorphic line bundles 
\[S\,\longrightarrow\, \Sigma\]
of degree $-2$ such that
\[S\otimes S\,=\,{\mathcal O}_\Sigma(-P_1-P_2-P_3-P_4)\, .\]
For every such $S$, there is a unique logarithmic connection $\nabla^S$ on $S$ with the property that
\[(\nabla^S\otimes\nabla^S) (s_{-P_1-P_2-P_3-P_4})\,=\,0\, ,\]
where $s_{-P_1-P_2-P_3-P_4}$ is the meromorphic section of
${\mathcal O}_\Sigma(-P_1-P_2-P_3-P_4)$ given by the constant function $1$ on $\Sigma$
(this section has simple poles at $P_1,\,\cdots,\,P_4$). The residue of $\nabla^S$ at $P_j$,
$1\, \leq\, j\, \leq\, 4$, is $\frac{1}{2}$.
Observe that the monodromy representation of $\nabla^S$ takes values in $\Z/2\Z.$
Also, note that $(S,\,\nabla^S)$ is unique up to tensoring with an order
two holomorphic line bundle $\xi$ equipped with the (unique) canonical connection that induces
the trivial connection on $\xi\otimes\xi$.

\begin{lemma}\label{trivialmon}
For given $\rho\,=\,\tfrac{1}{2p}$ and $\Sigma$ (see
\eqref{sigmayz}), consider $a^u$ and $\chi$ as in Theorem \ref{thetrivcon}.
There exists a unique pair $(S,\,\nabla^S)$ such that the monodromy of the connection
\[\nabla^S\otimes (\Pi\circ\Phi_p)^*\nabla^{a^u,\chi,\rho}\]
is trivial (see \eqref{abel-connection}, \eqref{mPi} and \eqref{zp}).
\end{lemma}

\begin{proof}
Since $p$ is odd,
$\rho\,=\, \tfrac{1}{2p}$, and $\Phi_p$ in \eqref{zp} is a totally branched covering,
the local monodromies of \[(\Pi\circ\Phi_p)^*\nabla^{a^u,\chi,\rho}\] around the points of
$P_i$, $1\,\leq\, i\, \leq\, 4$, are all $-\text{Id}.$ 

The totally branched covering $\Sigma\,\longrightarrow\,\mathbb CP^1$
in \eqref{sigmayz} is determined by its monodromy representation
\[M\,\colon\, \pi_1({\mathbb C}P^1\setminus\{\pm 1,\,\pm\sqrt{-1}\}, \, 0)\,\longrightarrow\, \mathcal S_{2p}\]
into the permutation group $\mathcal S_{2p}$ of the $2p$ points over $z\,=\,0$ which we label by
\[x\,\in\,\left\{\sqrt{-1},\, \sqrt{-1} \exp\left(\frac{\pi\sqrt{-1}}{p}\right),\, \cdots,\,
\sqrt{-1} \exp\left(\frac{\pi\sqrt{-1}j}{p}\right),\, \cdots,\, \sqrt{-1}
\exp\left(\frac{\pi\sqrt{-1}(2p-1)}{p}\right)\right\}\, . \]
The monodromy representation of the cyclic covering $\Sigma\,\longrightarrow\,\mathbb CP^1$
in \eqref{sigmayz} is
abelian, and the local monodromies around the 4 punctures are given by
\begin{equation}\label{Mmon}M_{1}\,=\, M_{-1} \,= \, \exp(\frac{\pi\sqrt{-1}}{p}),\ M_{\sqrt{-1}}\,=\, M_{-\sqrt{-1}}\,=\,\exp(-\frac{\pi\sqrt{-1}}{p})\, .
\end{equation}
The later can be computed via the logarithmic monodromy of $\log x$ by integrating
\[\frac{1}{2p}\frac{dx^{2p}}{x^{2p}}\,=\,\frac{1}{2p} \frac{d\tfrac{z^2-1}{z^2+1}}{\tfrac{z^2-1}{z^2+1}}\]
using the residue theorem.

The $p$--fold cyclic covering $\Phi_p$ in \eqref{zp}
is also determined by its monodromy representation
\[m\,\colon\, \pi_1(\widehat{T}^2\setminus\{p_1,\cdots,p_4\},\, [0])\,\longrightarrow\, \mathcal S_p\, .\]
As in \eqref{yz-coord}, $[0]\,\in\,{\widehat T}^2\,=\,\C/(2\Z+2\sqrt{-1}\Z)$ is a point lying over
$z\,=\,0$ with respect to \eqref{widehattyz}.
Again, the image
$m(\pi_1(\widehat{T}^2\setminus\{p_1,\cdots,p_4\},\, [0]))$ is abelian, and we claim that it is given by 
\begin{equation}\label{mmon}m_{p_1}\,=\, m_{p_3}\,=\, \exp\left(\frac{2\pi\sqrt{-1}}{p}\right),\,
m_{p_2}\,=\, m_{p_4}\,=\ \exp\left(-\frac{2\pi
\sqrt{-1}}{p}\right),\,
m_{\widehat\alpha}\,=\,1,\,m_{\widehat\beta}\,=\,1\, ;\end{equation}
here $m_{p_k}$ are the local monodromies around $p_k$.

The above claim simply follows by describing closed loops on 
the 4-punctured torus as special closed loops on the 4-punctured sphere and using \eqref{Mmon}.

Consider the unitary abelian monodromy representation
$$
R\, :\, \pi_1(\widehat{T}^2\setminus\{p_1,\,\cdots,\,p_4\},\, [0])\, \longrightarrow\, \mathrm{SU}(2)
$$ of the connection $\Pi^*\nabla^{a^u,\chi,\rho}$
on $\widehat{T}^2\setminus\{p_1,\,\cdots,\,p_4\}$ (see \eqref{abel-connection} and \eqref{mPi}). Using the
diagonal representation
\[u\,\colon\,\Z/p\Z\,\longrightarrow\,\mathrm{SU}(2),\ \ \exp\left(\frac{2\pi\sqrt{-1}}{p}\right)\,\longmapsto\,
\begin{pmatrix} \exp(\frac{2\pi\sqrt{-1}(k+1)}{p})&0\\0 &\exp(\frac{-2\pi\sqrt{-1}(k+1)}{p})\end{pmatrix}\] 
it follows from Theorem \ref{thetrivcon} and \eqref{mmon} that the two homomorphisms $u\circ m$ and $R$
from $\pi_1(T^2\setminus\{p_1,\,\cdots,\,p_4\},\, [0])$ to $\mathrm{SU}(2)$
differ only by a $\Z/2\Z$-representation
(with values in $\{\pm\text{Id}\}\subset\mathrm{SU}(2)$).
Note that the local monodromies of this $\Z/2\Z$-representation are $-\text{Id}$, and, as $p$ is odd,
the same holds for the corresponding $\Z/2\Z$-representation $\mu$ of the fundamental group
of the 4-punctured covering $\Sigma\setminus\Phi_p^{-1}\{p_1,\,\cdots,\,p_4\}.$

The spin bundle $S$ is then chosen to give the aforementioned $\Z/2\Z$-representation $\mu$ of the
fundamental group of the 4-punctured
surface $\Sigma\setminus \Phi_p^{-1}\{p_1,\,\cdots,\,p_4\}$. Finally, the lemma then follows from the fact
that the representation $m$ induces the trivial representation on $\Sigma\setminus \Phi_p^{-1}
\{p_1,\cdots,p_4\}$ by the standard property of the monodromy on a covering that it is the pullback of the
monodromy.
\end{proof}

Henceforth, we always assume that
$$
\rho\,=\,\tfrac{1}{2p}\, .
$$

The connection in Lemma \ref{trivialmon}
\[\nabla^S\otimes (\Pi\circ\Phi_p)^*\nabla^{a^u,\chi,\rho}\,=\,
\nabla^S\otimes (\Pi\circ\Phi_p)^*\nabla^{a^u,\chi,\tfrac{1}{2p}}\]
is defined on the vector bundle
\[S\otimes (L\oplus L^*)\, \longrightarrow\, \Sigma\, ,\]
where $L$ is the pull-back, by $\Pi\circ\Phi_p$, of the $C^\infty$ trivial line bundle
$T^2\times{\mathbb C}\, \longrightarrow \,T^2$ equipped with
Dolbeault operator \[\overline{\partial}+\chi d\overline{w}\,=\,
\overline{\partial}+\chi\cdot\overline{\partial w}\, .\]
For each $1\, \leq\,i\, \leq\, 4$, 
the residues of the connection $\nabla^S\otimes (\Pi\circ\Phi_p)^*\nabla^{a^u,\chi,\rho}$
at the point of $P_i\,=\, \Phi^{-1}_p(p_i)$ is
\begin{equation}\label{rescon}
\tfrac{1}{2}\begin{pmatrix} 1&-1\\-1&1\end{pmatrix}\end{equation}
with respect to a suitable frame at the points $P_i$ compatible with the decomposition
$S\otimes (L\oplus L^*)\, =\, (S\otimes L)\oplus (S\otimes L^*)$;
compare with Proposition \ref{explicit_coeff} and its proof.

As in \cite[\S~3]{He3}, there exists a holomorphic
rank two vector bundle $V$ on $\Sigma$ with trivial determinant, equipped with a holomorphic connection
$D$, together with a holomorphic bundle map
\begin{equation}\label{df}
F\,\colon\, S\otimes (L\oplus L^*)\, \longrightarrow\, V
\end{equation}
which is an isomorphism away from $P_1,\,\cdots ,\, P_4$,
such that
\begin{equation}\label{Ddef}\nabla^S\otimes (\Pi\circ\Phi_p)^*\nabla^{a^u,\chi,\rho}\,=\,F^{-1}\circ D\circ F\, .\end{equation}
From Lemma \ref{trivialmon} we know that $(V,\, D)$ is trivial.

\begin{lemma}\label{Lemired}
Assume $p\,\geq\, 3.$ Consider the strongly parabolic Higgs field
\[\Psi\,=\, \begin{pmatrix} dw&0\\0&-dw\end{pmatrix}\]
with respect to the parabolic structure induced by $\nabla^{a^u,\chi,\rho}$.
Then, 
\[\Theta\,= \, F\circ(\Pi\circ\Phi_p)^*\Psi\circ F^{-1}\]
is a holomorphic Higgs field on the trivial holomorphic vector bundle $(V,\, D^{0,1})
\,=\, (V,\, D'')$ (here the Dolbeault
operator for the trivial holomorphic structure is denoted by $D''$).
\end{lemma}

\begin{proof}
Consider the holomorphic Higgs field
\[(\Pi\circ\Phi_p)^*\Psi\,\colon\, S\otimes (L\oplus L^*)\,\longrightarrow\,
K_\Sigma\otimes S\otimes (L\oplus L^*)\]
on the rank two holomorphic bundle $S\otimes (L\oplus L^*).$ It vanishes 
of order $p-1\,\geq\,2$ at the singular points $P_1,\,\cdots ,\,P_4.$
Performing the local analysis (as in \cite[\S~3.2]{He3}) near $P_k$ of the normal form of the
homomorphism $F$ in \eqref{df},
we see that $$\Theta\,=\, F\circ(\Pi\circ\Phi_p)^*\Psi\circ F^{-1}$$ has no
singularities, i.e., it is a holomorphic Higgs field on
the trivial holomorphic vector bundle $(V,\, D'')$.
\end{proof}

\begin{theorem}\label{Main}
There exists a compact Riemann surface $\Sigma$ of genus $g\,>\,1$ with a irreducible holomorphic
connection $\nabla$ on the trivial holomorphic rank two vector bundle
${\mathcal O}^{\oplus 2}_\Sigma$ such that the image of the monodromy homomorphism for $\nabla$ is
contained in $\SL(2,\R)$.
\end{theorem}

\begin{proof}
For $\rho\,=\,\tfrac{1}{2p}$, with $p$ being an odd
integer, consider the connection $\nabla^{a,\chi,\rho}$, over rank two vector bundle on $T^2$,
given by Theorem \ref{real-mon-hatT2}. Since the image of the monodromy homomorphism
for $\Pi^*\nabla^{a,\chi,\rho}$ is conjugate to a subgroup of
$\SL(2,\R)$, and $\nabla^S$ has $\Z/2\Z$--monodromy, 
the image of the monodromy homomorphism for the connection
\[\nabla^S\otimes(\Pi\circ\Phi_p)^*\nabla^{a,\chi,\rho}\]
can be conjugated into $\SL(2,\R)$ as well. The same holds for the connection
\[\nabla\,:=\,F\circ(\nabla^S\otimes(\Pi\circ\Phi_p)^*\nabla^{a,\chi,\rho}) \circ F^{-1}\]
because $F$ is a (singular) gauge transformation. From Lemma \ref{Lemired} we know that 
$\nabla-D$ is a holomorphic Higgs field on the trivial holomorphic vector bundle $(V,\,D''),$
where $D$ is the trivial connection in \eqref{Ddef}.

It remains to show that the monodromy homomorphism for $\nabla$ is an irreducible
representation of the fundamental group. Since
$\rho\,\neq\,0$ is small, this follows from Lemma \ref{irr}. Indeed, observe
that there exists $\widetilde{\alpha},\, \widetilde{\beta}\,\in\,\pi_1(\Sigma,\,q)$
(see Figure \ref{figure3})
along which the monodromies of $\nabla$ are given by
\[h(\alpha)h(\alpha) \ \ \text{ and }\ \ h(\beta)h(\beta)\]
up to a possible sign. For example, representatives of $\widetilde{\alpha},\, \widetilde{\beta}$
are given by a connected component of the preimage of $\widehat{\alpha}$ and $\widehat{\beta}$ respectively.
Because $xy\,\neq\, 0$, in view of \eqref{T1eq} and \eqref{T2eq}
and continuity in $\rho$, the monodromy representation must
be irreducible by Lemma \ref{irr}.
\end{proof}

\section{Figures}

\begin{figure}[h]
\centering
\begin{subfigure}[t]{0.25\textwidth}
\centering
\includegraphics[width=\textwidth]{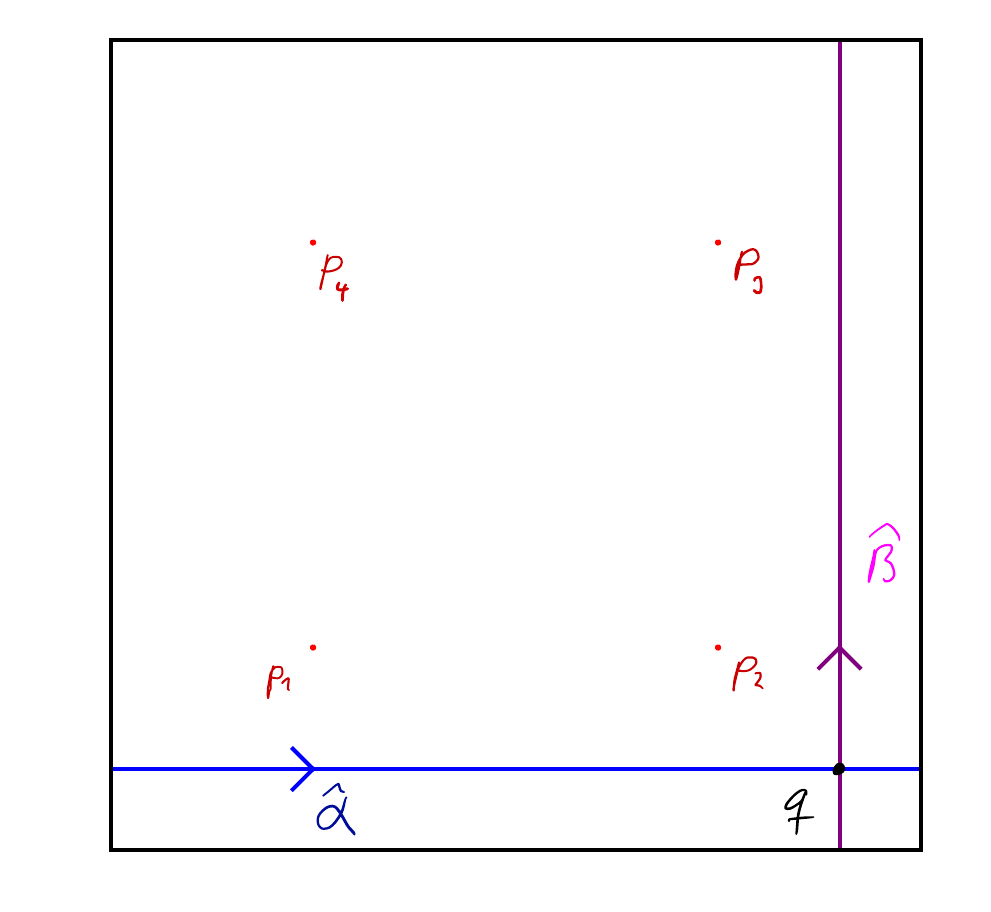}
\caption{
\small
The 4-punctured torus.
}
\label{figure1}
\end{subfigure}
\begin{subfigure}[t]{0.375\textwidth}
\centering
\includegraphics[width=\textwidth]{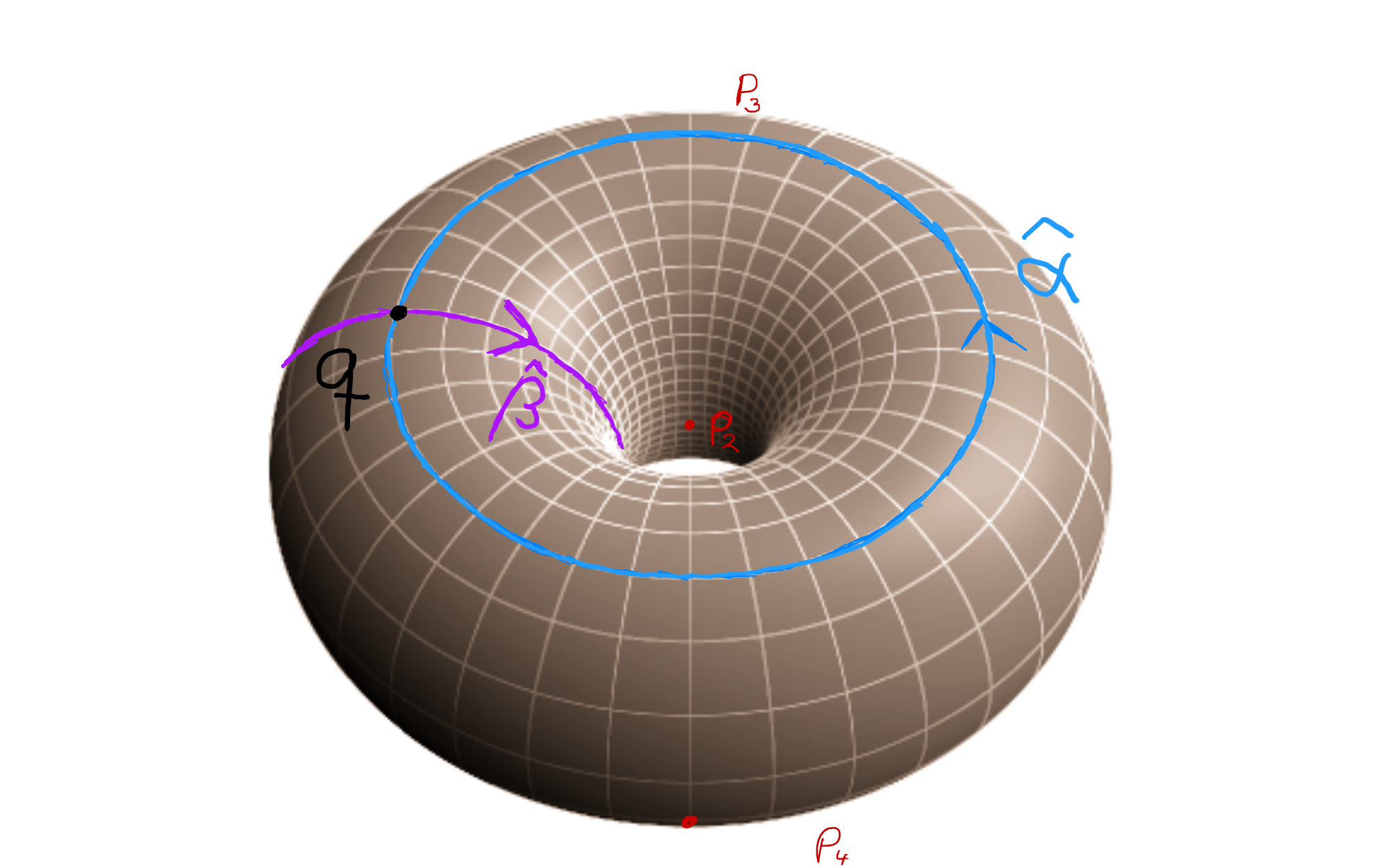}
\caption{\small
A second view of the 4-punctured torus.
}
\label{figure2}
\end{subfigure}
\begin{subfigure}[t]{0.275\textwidth}
\centering
\includegraphics[width=\textwidth]{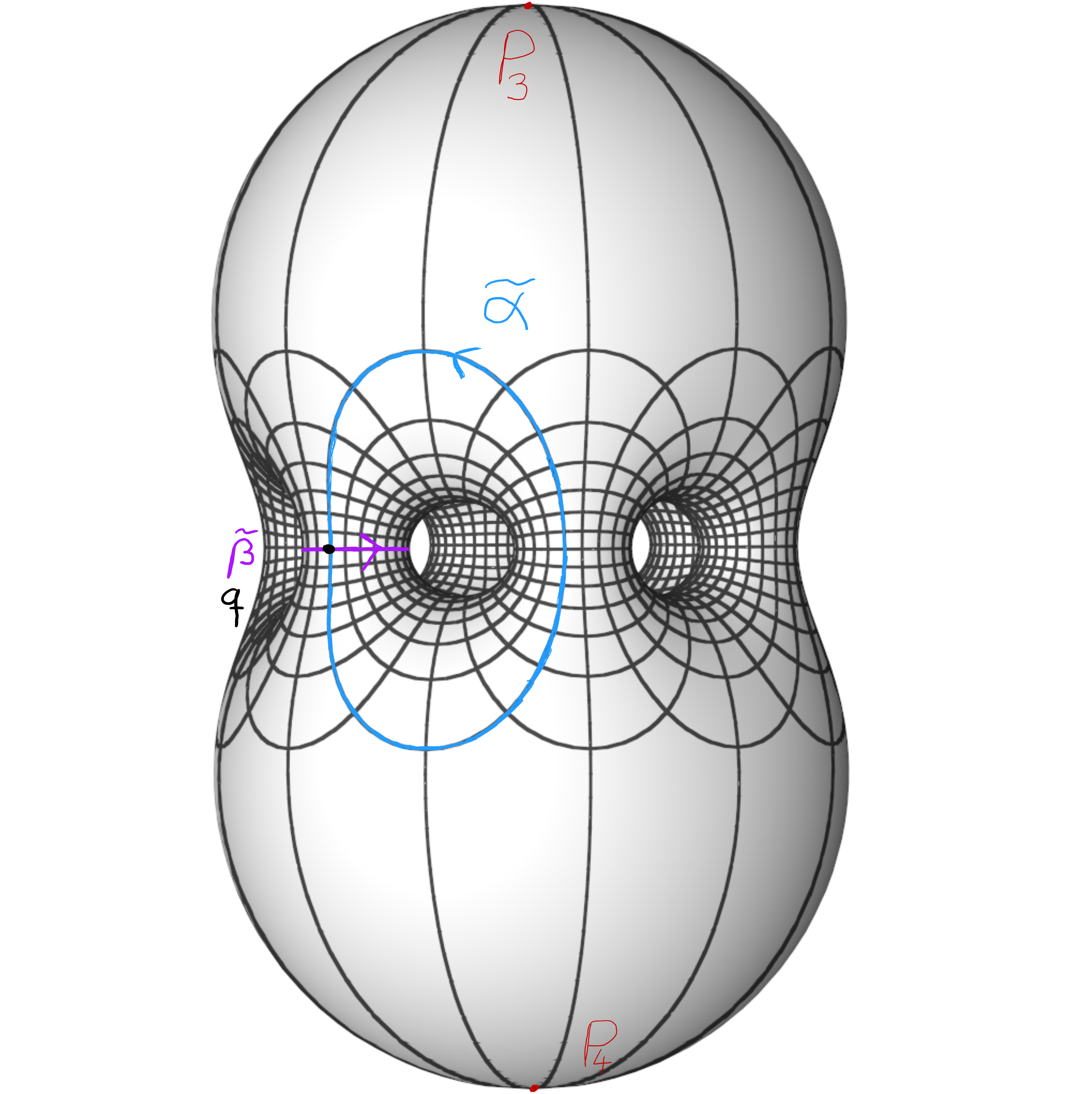}
\caption{\small
The Riemann surface $\Sigma$ for $p= 3$, shown with vertical and horizontal trajectories of
$(\Pi\circ\Phi_3)^*(dw)^2$. Picture by Nick Schmitt.
}
\label{figure3}
\end{subfigure}
\caption{}
\end{figure}

\section*{Acknowledgements}

We are grateful to the referees for helpful comments.
This work has been supported by the French government through the UCAJEDI Investments in the
Future project managed by the National Research Agency (ANR) with the reference number
ANR2152IDEX201. The first-named author is partially supported by a J. C. Bose Fellowship, and
school of mathematics, TIFR, is supported by 12-R$\&$D-TFR-5.01-0500.

\end{document}